\documentclass[a4paper,10pt]{article}
\usepackage[USenglish,british,american,australian,english]{babel}
\usepackage[utf8]{inputenc}
\usepackage[T1]{fontenc}
\usepackage{tikz}
\usepackage{amsmath,amsthm,amsfonts}
\usepackage{amssymb}

\usepackage{textcomp}
\usepackage{graphicx}

\usepackage{caption}
\usepackage{subcaption}

\usepackage{hyperref} 

\usepackage{tcolorbox} 

\usepackage{tikz} 
\graphicspath{{imagesstationary/}}

\newtheorem{theo}{Theorem}[section]

\newtheorem{lem}{Lemma}[section]
\newtheorem{prop}{Proposition}[section]
\newtheorem{ass}{Assumption}[section]
\theoremstyle{definition}

\theoremstyle{remark}

\newcommand\eps{\varepsilon}
\newcommand\de{\delta}
\newcommand\ph{\varphi}

\newcommand\la{\lambda}

\def \p {{\partial}}

\def \bu {\bar{u}}
\def \bq {\bar{q}}

 \def \fvm {\frac{\partial \bar{q}_2(0)}{\partial V_{m}}}
\def \fqk {\frac{\partial \bar{q}_2(0)}{\partial k}}
\usepackage{authblk}
\begin{document}

\title{On the role of the epithelium in a model of sodium exchange in renal tubules}
\author[1,2]{Marta Marulli}
\author[3]{Aur\'elie Edwards}
\author[2]{Vuk Mili\v{s}i\'{c}}
\author[2]{Nicolas Vauchelet}
\affil[1]{LAGA,  UMR 7539, CNRS, Universit\'e Sorbonne Paris Nord, 99, avenue Jean-Baptiste Clément 93430 Villetaneuse - France.}
\affil[2]{University of Bologna, Department of Mathematics, Piazza di Porta S. Donato 5, 40126 Bologna, Italy.}
\affil[3]{Department of Biomedical Engineering, Boston University, Massachusetts.}

\date{}                     
\setcounter{Maxaffil}{0}
\renewcommand\Affilfont{\itshape\small}
\maketitle

\begin{abstract}
In this study we present a mathematical model describing the transport of sodium in a fluid circulating in a counter-current tubular architecture, which constitutes a simplified model of Henle's loop  in a kidney nephron. 
The model explicitly takes into account the epithelial layer at the interface between the tubular lumen and the surrounding interstitium. In a specific range of parameters, we show that explicitly accounting for transport across the apical and basolateral membranes of epithelial cells, instead of assuming a single barrier, affects the axial concentration gradient, an essential determinant of the urinary concentrating capacity.
We present the solution related to the stationary system, and we perform numerical simulations to understand the physiological behaviour of the system. 
We prove that when time grows large, our dynamic model converges towards the stationary system at an exponential rate. 
In order to prove rigorously this global asymptotic stability result, we study eigen-problems of an auxiliary linear operator and its dual. \\

\end{abstract}



\section{Introduction}
One of the main functions of the kidneys is to filter metabolic wastes and toxins from plasma and excrete them in urine. The kidneys also play a key role in regulating the balance of water and electrolytes, long-term blood pressure, as well as acid-base equilibrium. The structural and functional units of the kidney are called nephrons, which number about 1 million in each human kidney \cite{bertram}.

 Blood is first filtered by glomerular capillaries and then the composition of the filtrate varies as it flows along different segments of the nephron~:  
 the proximal tubule, 
 Henle's loop (which is formed by a descending limb and an ascending limb), 
 the distal tubule, and 
 the collecting duct. The reabsorption of water and solutes from the tubules into the surrounding interstitium (or secretion in the opposite direction) allows the kidneys to match precisely urinary excretion to the dietary intake \cite{palmer}. \\

In the last decade, several groups have developed sophisticated models of water and electrolyte transport in the kidney. These models can be broadly divided into 2 categories: (a) detailed cell-based models that incorporate cell-specific transporters and predict the function of small populations of nephrons at steady-state (\cite{nieves}; \cite{weinstein}; \cite{laytonvallonedwards}; \cite{weinstein2017}; \cite{edwardsauberson}), and (b) macroscale models that describe the integrated function of nephrons and renal blood vessels but without accounting for cell-specific transport mechanisms (\cite{magali}, \cite{tesp}; \cite{mossthomas}; \cite{fry}; \cite{clemmer}; \cite{hallow}). These latter models do not consider explicitly the epithelial layer separating the tubule lumen from the surrounding interstitium, and represent the barrier as a single membrane. We developed the model presented below to assess the impact of this set of assumptions. \\
Specifically, in this study we present a simplified mathematical model of solute transport in Henle's loop. This model accounts for ion transport between the lumen and the epithelial cells, and between the cells and the interstitium. The aim of this work is to evaluate the impact of explicitly considering the epithelium on predicted solute concentration gradients in the loop of Henle.


In our simplified approach, the loop of Henle is represented as two tubules in a counter-current arrangement, the descending and ascending limb are considered to be rigid cylinders of length $L$ lined by a layer of epithelial cells. Water and solute reabsorption from the luminal fluid into the interstitium proceeds in two steps~: water and solutes cross first the apical membrane at the lumen-cytosol interface and then the basolateral membrane at the cytosol-interstitium interface, \cite{weinkra}. 
A schematic representation of the model is given in Figure \ref{fig:1}.

\begin{figure}	\centering
	\includegraphics[scale=0.45]{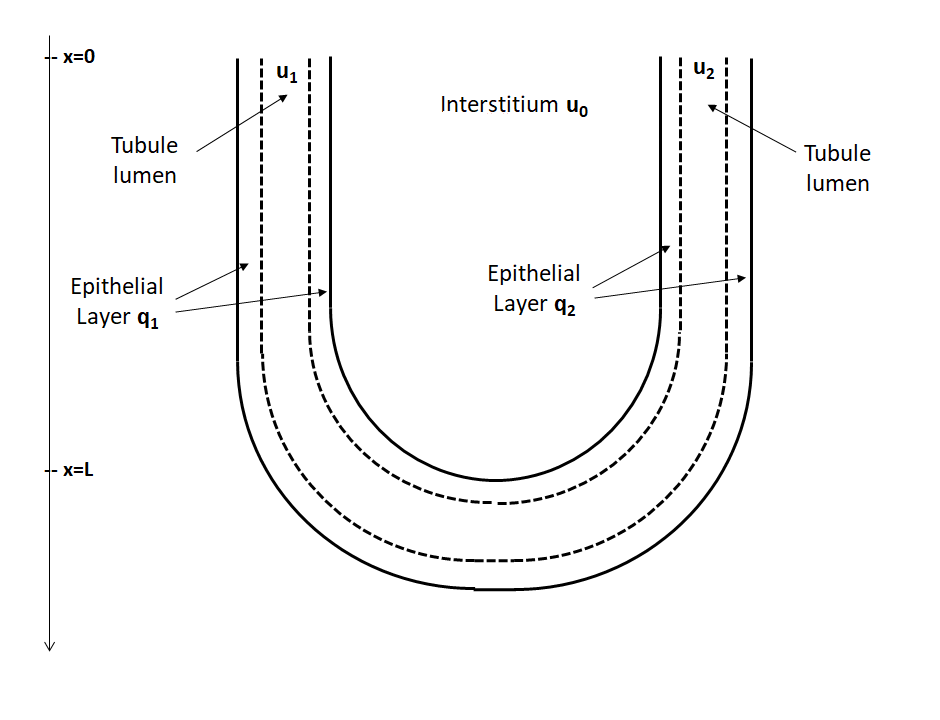}
	\caption{Simplified model of loop of Henle. $q_1/q_2$ and $u_1/u_2$ denote solute concentration in the epithelial layer and lumen of the descending/ascending limb, respectively.}
	\label{fig:1}
\end{figure}
The energy that drives tubular transport is provided by Na$^+$/K$^+$-ATPase, an enzyme that couples the hydrolysis of ATP to the pumping of sodium (Na$^+$) ions out of the cell and potassium (K$^+$) ions into the cell, across the basolateral membrane. The electrochemical potential gradients resulting from this active transport mechanism in turn drive the passive transport of ions across other transporters, via diffusion or coupled transport. We refer to diffusion as the biological process in which a substance tends to move from an area of high concentration to an area of low concentration \cite{shultz,thomas}.
As described in \cite{keener}, in the absence of electrical forces, the diffusive solute flux from compartment 1 to compartment 2 (expressed in $[mol.m^{-1}.s^{-1}]$) is given by:
$$J_{\text{diffusion}}=P \ell (u_{2}-u_{1}), $$
where $P \ [m.s^{-1}]$ is the permeability of the membrane to the considered solute, $\ell$ the perimeter of the membrane, and $u_{1}$ and $u_{2}$ are the respective concentrations of the solute in compartments $1$ and $2$.

We assume that the volumetric flow rate in the luminal fluid (denoted by $\alpha>0$) remains constant, i.e. there is no transepithelial water transport. 
The descending limb is in fact permeable to water, but we make this simplifying assumption in order to facilitate the mathematical analysis. Given the counter-current tubular architecture, this flow rate has a negative value in the ascending limb.

The present model focuses on tubular Na$^+$ transport.  The concentration of Na$^+$ ($[mol.m^{-3}$]) is denoted by $u_1$ and $u_2$, respectively, in the lumen of the descending and ascending limb, by $q_1$ and $q_2$ in the epithelial cells of the descending and ascending limbs, respectively, and by $u_0$ in the interstitium. The permeability to
Na$^+$ of the membrane separating the lumen and the epithelial cell of the descending and ascending limb is denoted by $P_1$ and $P_2$, respectively. $P_{1,e}$ denotes the permeability to Na$^+$ of the membrane separating the epithelial cell of the descending limb and the interstitium; the Na$^+$ permeability at the interface between the epithelial cell of the ascending limb and the interstitium is taken to be negligible.
The re-absorption or secretion of ions generates electrical potential differences across membranes. In the present model, the impact of transmembrane potentials on Na$^+$ transport is not taken into account.

The concentrations depend on the time $t$ and the spatial position $x\in[0,L]$. 
The dynamics of Na$^+$ concentration is given by the following model on $(0,+\infty)\times (0,L)$
\begin{align}
&a_1 \frac{\partial u_1}{\partial t}+\alpha\frac{\partial u_1}{\partial
x}=J_1, 
\quad a_2 \frac{\partial u_2}{\partial t}-\alpha\frac{\partial u_2}{\partial
x}=J_2, \label{problem1} \\
&a_3\frac{\partial q_{1}}{\partial t}=J_3, 
\quad a_4 \frac{\partial q_{2}}{\partial t}=J_4,
\quad a_0 \frac{\partial u_{0}}{\partial t}=J_0.  \label{problem2} 
\end{align}
The parameters $a_i$, for $i=0,1,2,3,4$, denote positive constants defined as: 
$$
a_1=\pi r_1^2, \ a_2=\pi r_2^2, \ a_3=\pi (r_{1,e}^2-r_{1}^2), \ a_4=\pi( r_{2,e}^2-r_{2}^2), \ a_0=\pi\Bigl( \frac{r_{1,e}^2+r_{2,e}^{2}}{2}\Bigl).
$$
In these equations, $r_i$, $i=1,2$, denotes the inner radius of tubule $i$, whereas $r_{i,e}$ denotes the outer radius of tubule $i$, which includes the epithelial layer.
The fluxes $J_i$ describe the ionic exchanges between the different domains.
They are modeled in the following way:
\begin{description}
 \item \textbf{Lumen.} In the lumen, we consider the diffusion of Na$^+$ towards the epithelium. Then,
$$
J_1=2\pi r_1 P_{1} (q_{1}-u_1), \quad J_2=2\pi r_2 P_{2} (q_{2}-u_2).
$$
\item \textbf{Epithelium.} We take into account the diffusion of Na$^+$ from the descending limb epithelium towards both the lumen and the interstitium,
$$
J_3=2\pi r_1 P_{1} (u_1-q_{1}) + 2\pi r_{1,e} P_{1,e} (u_{0}-q_{1}).
$$
In the ascending limb (tubule 2), we also consider the active
reabsorption that is mediated by Na$^+$/K$^+$-ATPase, which pumps 3 Na$^+$ ions out of the cell in exchange for 2 K$^+$ ions. 

The net flux into the ascending limb epithelium is given by the sum of the diffusive flux from the lumen and the export across the pump, which is described using Michaelis-Menten kinetics, \cite{keener}:
$$
J_4=2\pi r_2 P_{2} (u_2-q_{2}) -2\pi r_{2,e} G(q_{2}),
$$
where
$$
G(q_{2}) = V_{m}\left[\frac{q_{2}}{K_{M,2}+q_{2}}\right]^3.
$$
The exponent of $G$ is related to the number of exchanged sodium ions.
The affinity of the pump $K_{M,2}$, and its maximum velocity $V_{m}$, are given real numbers.
We notice that when $q_2 \rightarrow +\infty$, then $G(q_2)\rightarrow  V_{m}$ which is in accordance with the biological observation that the pump can be saturated.
\item \textbf{Interstitium.} 
$$
J_0=2\pi r_1 P_{1,e} (q_{1}-u_{0}) + 2\pi r_{2,e} G(q_{2}).
$$
\end{description}
We model the dynamics of a solute (here sodium) by the evolution of its concentration in each tubule. The transport of solute and its exchange are then modelled by a hyperbolic PDE system at constant speed with a non-linear transport term and with specific boundary conditions.

The dynamics of ionic concentrations is given by the following model:
\begin{equation}\label{problem}
\begin{cases}
a_1 \p_{t}u_1(t,x) + \alpha\p_{x} u_1(t,x)=J_1(t,x) \\
a_2 \p_{t}u_2(t,x) - \alpha\p_{x} u_2(t,x)=J_2(t,x) \\
a_3 \p_{t}q_{1}(t,x)=J_{3}(t,x) \\
a_4 \p_{t}q_{2}(t,x)=J_{4}(t,x) \\
a_0 \p_{t}u_0(t,x)=J_0(t,x).
\end{cases}
\end{equation}
We set the boundary conditions~:
\begin{equation}\label{bond}
u_1(t,0)=u_b(t), \quad   u_2(t,L)=u_1(t,L), \quad t > 0,
\end{equation}
where $u_b$ is a given function in $L^{\infty}(\mathbb{R}^{+})\cap L^{1}_{loc}(\mathbb{R}^{+})$, which is such that $\lim_{t \rightarrow \infty} u_b(t) =\bu_b$ for some positive constant $\bu_b>0$.

Finally, the system is complemented with initial conditions
\begin{align*}
  &u_1(0,x) = u_1^0(x),\quad u_2(0,x) = u_2^0(x),\quad u_0(0,x) = u_0^0(x),\\
  &q_1(0,x) = q_1^0(x), \quad q_2(0,x) = q_2^0(x).
\end{align*}
To simplify notations in \eqref{problem}, we set $K_1:=2\pi r_{1,e}P_{1,e}$, $k_1:=2\pi r_1 P_1$, and $k_2:=2\pi r_2 P_2$. For the diffusive fluxes $J_4$ and $J_0$, we include the constant $2\pi r_{2,e}$ in the parameter $V_{m}$ and replace the parameter $V_m$ with $V_{m,2} := 2\pi r_{2,e} V_m$, such that
\begin{equation}\label{defG}
G(q_{2}) = V_{m,2}\left[\frac{q_{2}}{K_{M,2}+q_{2}}\right]^3.
\end{equation}
Moreover, the orders of magnitude of $k_1, k_2$ are the same even if their values are not definitely equal, we may assume to further simplify the analysis that $k_1=k_2=k$.
We will refer to this as the dynamic system and then \eqref{problem} reads~:
\begin{subequations}\label{five}
  \begin{equation}\label{fivea}
    a_1 \partial_{t}u_1 + \alpha\partial_{x} u_1=k(q_1-u_1) 
  \end{equation} 
  \begin{equation}\label{fiveb}
    a_2 \partial_{t}u_2 - \alpha\partial_{x} u_2=k(q_2-u_2) 
  \end{equation} 
  \begin{equation}\label{fivec}
    a_3 \partial_{t}q_1=k(u_{1}-q_1) + K_1(u_{0}-q_1) 
  \end{equation} 
  \begin{equation}\label{fived}
    a_4 \partial_{t}q_2=k(u_{2}-q_2) - G(q_2)
  \end{equation} 
  \begin{equation}\label{fivee}
    a_0 \partial_{t}u_0=K_1(q_1-u_0) + G(q_2).
  \end{equation}
\end{subequations}
The existence and uniqueness of vector solution $\mathbf{u}=(u_1, u_2, q_1, q_2, u_0)$ to this system are investigated in \cite{MMV}.
 Several previous works have neglected the epithelium region (see e.g. \cite{tesp,magali}).
 The first goal of this work is to study the effects of this region in the mathematical model. The main indicator quantifying these effects is the parameter $k$ which accounts for the permeability between the lumen and the epithelium. Then, we analyse the dependency between the concentrations and $k$.
In the absence of physiological perturbations, the concentrations are very close to the steady state, thus it seems reasonable to consider solutions of \eqref{five}  at equilibrium, which leads us to study the system~:
\begin{equation}\label{stat}
\begin{cases}
+\alpha\p_x \bu_1 = k(\bq_1-\bu_1) \\
- \alpha\p_x \bu_2= k (\bq_2-\bu_2) \\
0= k(\bu_1-\bq_1) +K_1 (\bu_0-\bq_1) \\
0= k(\bq_2-\bu_2) - G(\bq_2) \\
0= K_1(\bq_1-\bu_0) + G(\bq_2)\\
\bu_1(L)=\bu_2(L), \quad \bu_1(0)=\bu_b.
\end{cases}
\end{equation}
Section \ref{sec:stat} concerns the analysis of solutions to stationary system \eqref{stat}. In particular, we study their qualitative behaviour and their dependency with respect to the parameter $k$. Our mathematical observations are illustrated by some numerical computations.

The second aim of the paper is to study the asymptotic behaviour of the solutions of \eqref{five}.
In Theorem \ref{timeconvergence}, we show that they converge as $t$ goes to $+\infty$ to the steady state solutions solving  \eqref{stat}.
Section \ref{sec:longtime} is devoted to the statement and the proof of this convergence result.
Finally, an Appendix provides some useful technical lemmas.

\section{Stationary system}
\label{sec:stat}

In this section, after proving basic existence and uniqueness results,  
we investigate how solutions of \eqref{stat} 
depend upon the parameter $k$. We recall that it includes also the permeability parameter as $k=k_i$ with $k_i:=2\pi r_i P_i$, $i=1,2$.
%
%
In order to study the qualitative behaviour of these solutions, we then perform some numerical simulations. 

\subsection{Stationary solution}

We first show existence and uniqueness of solutions to the stationary system:
\begin{lem}\label{lem:stat}
  Let $\bu_b>0$.
  Let $G$ be a $C^2$ function, uniformly Lipschitz, such that $G''$ is uniformly bounded and $G(0)=0$ (e.g. the function defined in \eqref{defG}).
  Then, there exists an unique vector solution to the stationary problem \eqref{stat}.

   Moreover, if we assume that $G>0$ on $\mathbb{R}^+$, then we have the following relation
  $$
  \bq_2 < \bu < \bq_1 < \bu_0.
  $$
\end{lem}

\begin{proof}
Summing up all the equations of system \eqref{stat}, we deduce that $\alpha ( \p_x \bu_1 - \p_x \bu_2 )= 0 $. From the boundary condition $\bu_1(L)=\bu_2(L)$, we obtain $\bu_1=\bu_2=\bu$. Therefore, we may simplify system \eqref{stat} in
\begin{equation}\label{eq}
 \begin{cases}
\alpha\p_{x} \bu=k(\bu-\bq_2) \\
2\bu=\bq_1+\bq_2 \\
0=k(\bu-\bq_{1}) + K_{1}(\bu_{0}-\bq_{1}) \\
0=k(\bu-\bq_{2}) - G(\bq_{2})\\
0=K_{1}(\bq_1-\bu_0) + G(\bq_2).
\end{cases}
\end{equation}
By the fourth equation of \eqref{eq},  $ \bu=\bq_2+\frac{G(\bq_2)}{k}$, inserted into the first equation, it gives 
 $ \p_x \bu =\frac{G(\bq_2)}{\alpha}$.
We obtain a differential equation satisfied by $\bq_2$,
\begin{equation}\label{reducedsystem}
\p_x \bq_2 = \frac{  G(\bq_2)}{\left(\alpha +\frac{\alpha }{k} G'(\bq_2)\right)},
\end{equation}
 with $\alpha, k$ positive constants and provided with the initial condition $\bq_2(0)$  that  satisfies
 \begin{equation}
   \label{eq:q20}
   \bq_2(0) + \frac{G(\bq_2(0))}{k} = \bu_b.
 \end{equation}
 We first remark that $\bq_2(0) \mapsto \bq_2(0) + \frac{G(\bq_2(0))}{k}$ is a $C^2$ increasing function which takes the value $0$ at $0$ and goes to $+\infty$ at $+\infty$.
 Thus, for any $\bu_b>0$ there exists a unique $\bq_2(0)>0$ solving \eqref{eq:q20}. 
 
 By assumption, $G'$ and $G''$ are uniformly bounded, thus we check easily that the right-hand side of \eqref{reducedsystem} is uniformly Lipschitz. Therefore, the Cauchy problem \eqref{reducedsystem}--\eqref{eq:q20} admits a unique solution, which is positive (by uniqueness since $0$ is a solution).

 Then, other quantities are computed thanks to the relations:
\begin{equation}\label{concentrations}
  \bu=\bq_2+\frac{G(\bq_2)}{k}, \quad
  \bq_1=\bq_2+\frac{2 G(\bq_2)}{k}, \quad
  \bu_0=\Bigl(\frac{1}{K_{1}}+\frac{2}{k}\Bigr)G(\bq_2)+\bq_2.
\end{equation}

Moreover, by the fourth and fifth equations of system \eqref{eq} and since $G(\bq_2)>0$, we immediately deduce that $\bq_2<\bu$ and $\bq_1<\bu_0$.
Using the second equation of \eqref{eq}, we obtain the claim.
\end{proof}

\subsection{Numerical simulations of stationary solutions}
\label{subsec:num}

\begin{center}
  \begin{table}
    \begin{tabular}{ l | c | r }
      \hline
      Parameters & Description & Values \\ \hline
      L & Length of tubules & $ 2 \cdot 10^{-3} \ [m]$ \\ 
      $\alpha$ & Water flow in the tubules & $10^{-13}  \ [m^3/s]$  \\
      $r_i$ &  Radius of tubule $i=1,2$ & $10^{-5} \ [m]$  \\
      $r_{i,e}$ &  Radius of epithelium layer $i=1,2$ & $1.5 \cdot 10^{-5} \ [m]$  \\
      $ K_1 $  & $2\pi r_{1,e} P_{1,e}$  &  $\sim 2 \pi \cdot 10^{-11} \ [m^2/s]$ \\
   	  $k=k_i $  & $2\pi r_i P_{i}, \ i=1,2$ & changeable $ [m^2/s]$ \\
    
      $V_{m,2}$ & Rate of active transport & $\sim 2\pi r_{2,e}10^{-5} \ [mol.m^{-1}.s^{-1}]$\\
      $K_{M,2}$ & Pump affinity for sodium ($Na^{+}$) & $3,5 \ [mol/m^3]$ \\
      $\bu_b$ & Initial concentration in tubule $1$  & $140 \ [mol/m^3]$  \\
      \hline
    \end{tabular}
    \caption{Frequently used parameters \label{table1}}
  \end{table}
\end{center}
We approximate numerically solutions of \eqref{eq}.
Numerical values of the parameters (cf Table \ref{table1}) 
are extracted from Table 2 in \cite{hervy} and Table 1 in \cite{laytonlayton}.

Taking into account these quantities allow us to have the numerical ranges of the constants and the solution results in a biologically realistic framework.
Following the proof of Lemma \ref{lem:stat}, 
we first solve \eqref{eq:q20} thanks to a Newton method. 
Then, we solve \eqref{reducedsystem} with a fourth order Runge-Kutta method.
Finally, we deduce other concentrations $u, q_1, u_0$ using  \eqref{concentrations}.
\begin{figure}[htbp]
  \centering
  \begin{subfigure}[b]{0.45\textwidth} 
    \centering \includegraphics[width=1.1\textwidth]{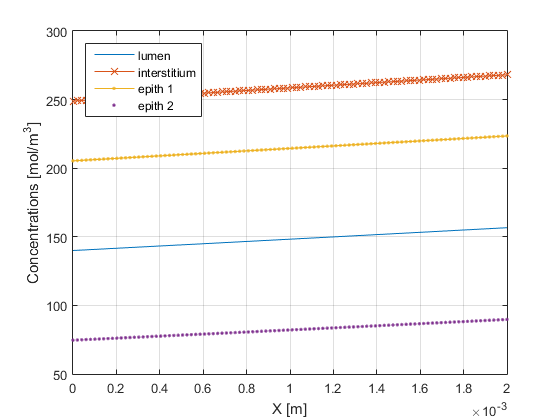}
    \caption{Concentration profiles with permeability $P_i=2\cdot 10^{-7}$ [m/s].}\label{fig:conc7}
  \end{subfigure}
  \begin{subfigure}[b]{0.45\textwidth}
    \centering \includegraphics[width=1.1\textwidth]{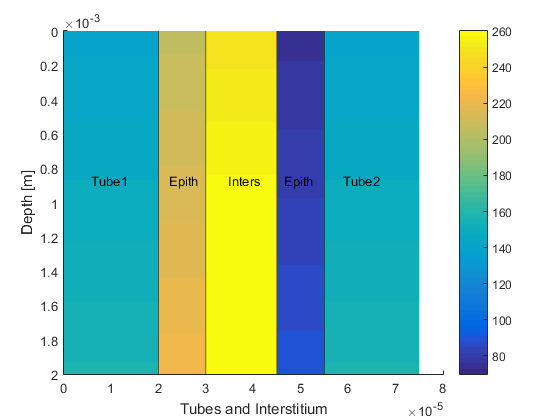}
    \caption{Concentrations in 2D  with $P_i=2\cdot 10^{-7}$. Length of lumen on vertical axis.}\label{fig:tube7}
  \end{subfigure}
  
  \begin{subfigure}[b]{0.45\textwidth}
    \centering \includegraphics[width=1.1\textwidth]{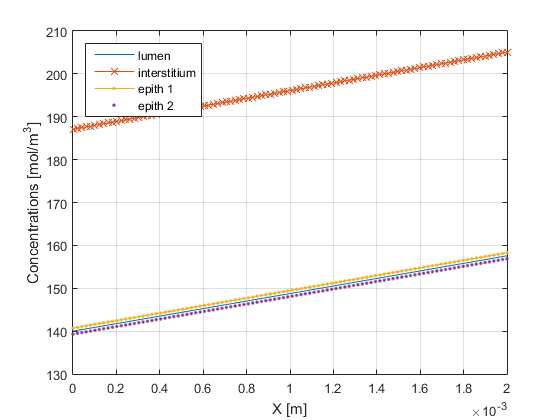}
    \caption{Concentration profiles with permeability $P_i=2\cdot 10^{-5}$ [m/s].}\label{fig:conc5}
  \end{subfigure}
  \begin{subfigure}[b]{0.45\textwidth}
    \centering \includegraphics[width=1.1\textwidth]{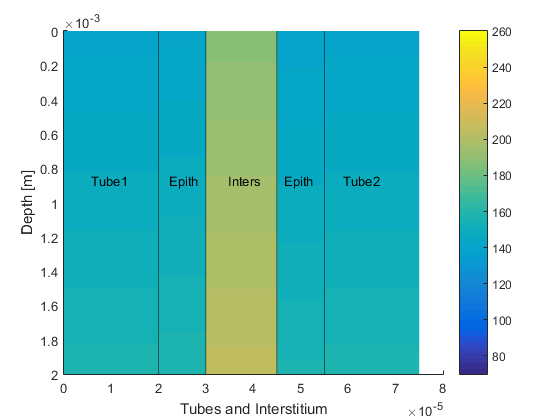}
    \caption{Concentrations in 2D  with $P_i=2\cdot 10^{-5}$. Length of lumen on vertical axis.}\label{fig:tube5}
  \end{subfigure}
  \caption{Concentration profiles for $V_{m,2}=2\pi r_{2,e} 10^{-5}$ and different permeability values.}\label{fig:profiles}
\end{figure}
        
Results from Figures \ref{fig:conc7} and \ref{fig:conc5} show that in all compartments, concentrations increase as a function of depth ($x$-axis).
Physiologically, this means that the fluid is more concentrated towards the hairpin turn ($x=L$) than near $x=0$, because of active transport in the ascending limb. 
It can also be seen that Na$^+$ concentration is higher in the central layer of interstitium and lower 
in the ascending limb epithelium owing to active Na$^+$ transport from the latter to the central compartment, 
described by the non-linear term $G(q_2)$. 
Furthermore, Figure \ref{fig:tube7} and Figure \ref{fig:tube5} highlight that increasing the permeability 
value homogenizes the concentrations in the tubules and in the epithelium region. 
Taking a very large permeability value is equivalent to fusing the epithelial layer with the adjacent lumen, 
such that luminal and epithelial concentrations become equal.
It is proved rigorously in \cite{MMV} that this occurs in the dynamic system \eqref{five}.
This is derived and explained  formally in Appendix \ref{formalcomputation}.

\newcommand{\FIC}{{\rm FIC}}
Figures \ref{fig:pompe} and \ref{fig:pompe2} depict the impact of permeability $P_1=P_2=P$ on concentration profiles for various  pump rates $V_{m,2}$. 
 Axial profiles of luminal concentrations are 
shown in Figures \ref{fig:testv5} and \ref{fig:testv4}, considering different values of the 
permeability between the lumen and the epithelium. The fractional increase in concentration (FIC)
is shown in Figures \ref{fig:percentagev5} and \ref{fig:percentagev4}~: for each permeability value (plotted on the horizontal axis), we compute the following ratio (shown on the vertical axis):
\begin{equation}\label{percentage}
\FIC(\bu):=100\frac{\bu(L)-\bu(0)}{\bu(0)},
\end{equation}
where  $\bu(L) $ is the concentration in the tubular lumen $1,2$ at $x=L$ and $\bu(0)$ the concentration at $x=0$. 
This illustrates the impact of permeability on the axial concentration gradient.
We observe that this ratio depends also strongly on the value of $V_{m,2}$.



	\begin{figure}[htbp]
	\centering
	\begin{subfigure}[b]{0.45\textwidth}
	\centering \includegraphics[width=1.1\textwidth]{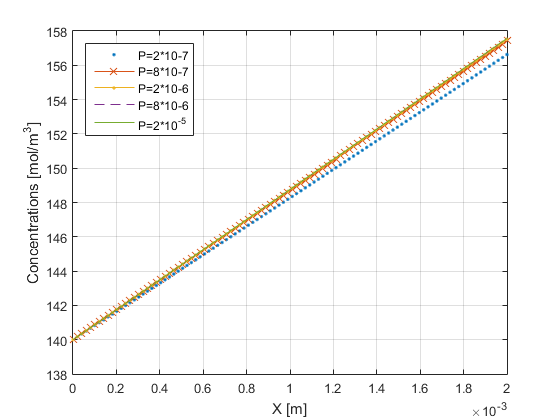}
	\caption{Axial concentrations in the lumen for different values of permeability.}\label{fig:testv5}
	\end{subfigure}
	\begin{subfigure}[b]{0.45\textwidth}
	\centering \includegraphics[width=1.1\textwidth]{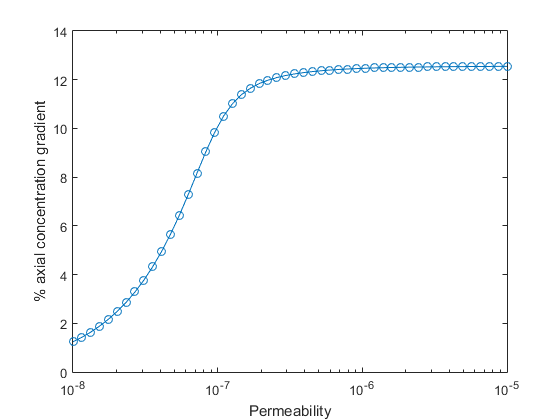}
	\caption{Fractional increase in concentration as a function of permeability.}\label{fig:percentagev5}
	\end{subfigure}
	\caption{Concentration profiles for  $V_{m,2}= 2\pi r_{2,e}\cdot 10^{-5} $  $[mol.m^{-1}.s^{-1}]$.}\label{fig:pompe}
	\end{figure}


    The permeability range (numerically $P \in [10^{-8}, 10^{-5}]$, equispaced 50 values between these) encompasses the physiological value which should be around  $10^{-7}$\ m/s.  As shown in Figures \ref{fig:percentagev5} and \ref{fig:percentagev4}, the FIC increases significantly with $P$ until it reaches a plateau~: indeed, as diffusion becomes  
        more rapid than  active transport (that is, pumping by Na$^+$/ K$^+$-ATPase), the permeability ceases to be rate-limiting. As shown by comparing Figures \ref{fig:percentagev5} and \ref{fig:percentagev4}, the FIC is strongly determined by the pump rate $V_{m,2}$~:
if $P \in [10^{-8}, 10^{-6}]$, Na$^+$ concentration along the lumen increases by less than $12\%$ if $V_{m,2}=2 \pi r_{2,e} 10^{-5}$, and may reach $120 \%$ if $V_{m,2}=2 \pi r_{2,e} 10^{-4}$. 
This raise is expected since concentration differences are generated by active transport~; the higher the rate of active transport, the more significant these differences. Conversely, in the absence of pumping, concentrations would equilibrate everywhere. 
In regards to the axial gradient, the interesting numerical results are in Figure \eqref{fig:testv4} and \eqref{fig:percentagev4}. We observe that the axial gradient increases with increasing permeability when the latter is varied within the chosen range.
Therefore this indicates that taking into account the epithelial layer in the model has a significant influence on the axial concentration gradient.
\\

Moreover, numerical results also confirm that~: $\bu_1=\bu_2=u<\bq_1<\bu_0$ as reported in Lemma \ref{lem:stat}.
We recall that we assume a constant water flow $\alpha$ which allows us to deduce $\bu_1 = \bu_2$. As noted above, the descending limb is in fact very permeable to water and $\alpha$ should decrease significantly in this tubule, such that $\bu_1$ differs from $\bu_2$, except at the hairpin turn at $x = L$. On the other hand, the last equation of system \eqref{eq} implies that $\bq_1 < \bu_0$, meaning that the concentration of Na$^+$ is lower in the epithelial cell than in the interstitium, as observed in vivo, \cite{atherton}.

With the expression of $G$ in \eqref{defG}, equation \eqref{eq:q20} reads
\begin{equation}\label{initialq2}
  \bq_2(0) + \frac{ V_{m,2}}{k}  \left( \frac{\bq_2(0)}{K_{M,2}+\bq_2(0)} \right)^{3}=\bu_b.
\end{equation}
In order to better understand the behaviour of the axial concentration gradient shown in Figures \eqref{fig:percentagev5}, \eqref{fig:percentagev4}, we compute the derivative of \eqref{initialq2} with respect to the parameter $V_m$ and with respect to $k$ respectively: 
$$
\fvm  +\frac{1}{k} G'(\bq_2(0)) \fvm  + \frac{1}{k} \bigl( \frac{\bq_2(0)}{K_{M,2}+\bq_2(0)}\bigr)^3 =0,
$$
$$
\fqk  +\frac{1}{k} G'(\bq_2(0)) \fqk  - \frac{1}{k^2} G(\bq_2(0))=0.
$$
Then, we get
$$
\fvm  =\frac{-\frac{1}{k}}{1+\frac{1}{k} G'(\bq_2(0))}   \bigl( \frac{\bq_2(0)}{k_{M}+\bq_2(0)}\bigr)^3 \leq 0 ,
$$
$$
\fqk   =\frac{\frac{1}{k^2} G(\bq_2(0))}{1+\frac{1}{k} G'(\bq_2(0))} \geq 0,
$$
because $G$ is a monotone non-decreasing function and 
$q_2(0)$ is positive.
 
We observe from numerical results  (see Figures \ref{fig:conc7}, \ref{fig:conc5}, \ref{fig:testv5}, and \ref{fig:testv4}) 
that the gradient of $u$ is almost constant.
Thus, we may make the approximation
\begin{equation}
  \label{eq:approxpu}
  \p_x \bu \sim \p_x \bu(0) = \frac{G(\bq_2(0))}{\alpha}.  
\end{equation}
Its derivatives with respect to $V_{m,2}$ and $k$ are both non negative~:
$$
\frac{\p}{\p k}[\p_x \bu] \sim \frac{G'(\bq_2(0))}{\alpha} \Bigl(\frac{\frac{G(\bq_2(0))}{k^2}}{1+\frac{1}{k}G'(\bq_2(0))}  \Bigr) \geq 0,
$$
$$
\frac{\p}{\p V_{m}}[\p_x \bu] \sim \frac{1}{\alpha} \Bigl(\frac{\bq_2(0)}{K_{M}+\bq_2(0)}\Bigr)^3 \Bigl(\frac{1}{1+\frac{1}{k}G'(\bq_2(0))}  \Bigr) \geq 0.
$$
It means that the axial concentration gradient is an increasing function both with respect to the rate of active transport $V_{m,2}$ and to the permeability $k$.

\begin{figure}[htbp]
	\centering
	\begin{subfigure}[b]{0.45\textwidth}
		\centering \includegraphics[width=1.1\textwidth]{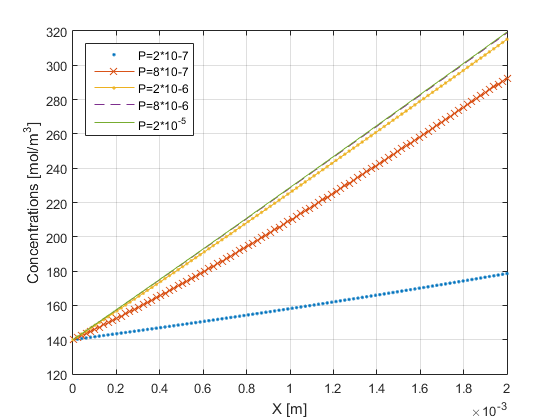}
		\caption{Axial concentrations in the lumen for different values of permeability.}\label{fig:testv4}
	\end{subfigure}
	\begin{subfigure}[b]{0.45\textwidth}
		\centering \includegraphics[width=1.1\textwidth]{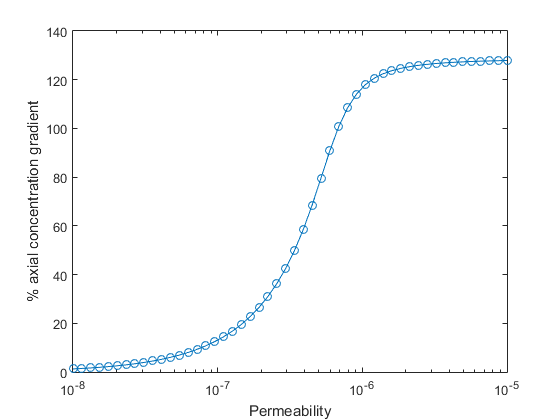}
		\caption{Percentage of concentration gradient as a function of permeability.}\label{fig:percentagev4}
	\end{subfigure}
	\caption{Concentration profiles for  $V_{m,2}= 2\pi r_{2,e}\cdot 10^{-4} $  $[mol.m^{-1}.s^{-1}]$}\label{fig:pompe2}
\end{figure}


\begin{figure}[htbp]
	\centering
	\begin{subfigure}[b]{0.45\textwidth}
		\centering \includegraphics[width=1.1\textwidth]{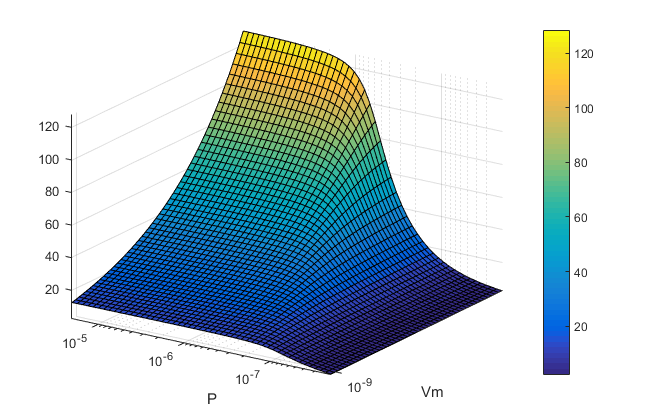}
		\caption{Percentage of concentration gradient 2D in tubules  }\label{fig:percentage2d}
	\end{subfigure}
	\begin{subfigure}[b]{0.45\textwidth}
		\centering \includegraphics[width=1.1\textwidth]{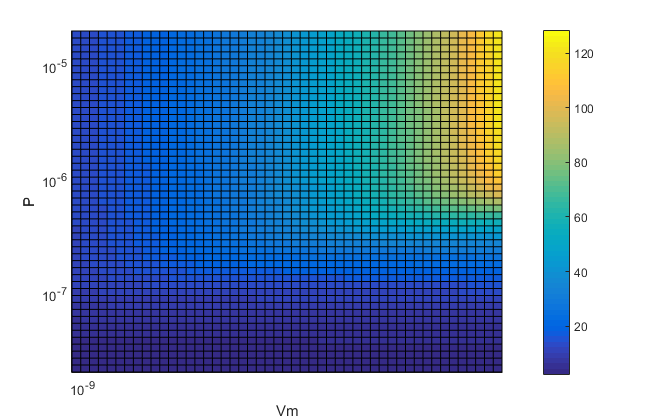}
		\caption{Percentage of concentration gradient projection}\label{fig:percentage2d2}
	\end{subfigure}
	\caption{Percentage of concentration gradient 2D with range $V_{m,2} \in 2\pi r_{2,e}\cdot( 10^{-5}, 10^{-4}) $  $[mol.m^{-1}.s^{-1}]$ ($x-$axis) and $P \in 2 \cdot(10^{-8}, 10^{-5})$ ($y-$axis)}\label{fig:}
\end{figure}
Indeed in Fig. \ref{fig:}, we perform numerical simulations varying both $P$ and $V_{m,2}$ and observe that  the ratio \eqref{percentage} increases monotonically with respect to both parameters.

\subsection{Limiting cases: $k \rightarrow \infty $, and $k \rightarrow 0$}

Numerical results show that above a certain high value of permeability, 
the epithelial concentration in tubule 2 seems to reach a plateau (see Figures \ref{fig:percentagev5} and \ref{fig:percentagev4}).
There are two different regimes : one for large values of permeabilities, one for small values of permeabilities, and a fast transition between them.

In the large permeabilities asymptotic, we may approximate system \eqref{eq} by the limiting model $k= +\infty$.
In this case, \eqref{concentrations} reduces to
\begin{equation*}
\bu=\bq_2, \quad \bq_1=\bq_2, \quad \bu_0=\frac{G(\bq_2)}{K_1}+\bq_2.
\end{equation*}
for all $x \in (0,L)$.
This is understandable from a formal point of view, also taking into account 
computations in Appendix \eqref{formalcomputation} for the stationary system \eqref{stat}.
In this case, the gradient concentration is directly proportional to $V_{m,2}$~:
\begin{equation*}
\p_x \bq_2 =\frac{G(\bq_2)}{\alpha + \frac{\alpha}{k} G'(\bq_2)} \underset{k\to +\infty}{\longrightarrow} \frac{G(\bq_2)}{\alpha}=\p_x \bu.
\end{equation*}
From \eqref{reducedsystem}, the Cauchy problem reduces to
\begin{equation*}
\p_x \bq_2(x)=\frac{G(\bq_2)}{\alpha}, \quad \bq_2(0)=\bu_b.
\end{equation*}
Additionally, it is clear that the higher pump value, the more the FIC  will increase, as observed in Figure \ref{fig:percentagev4}.

On the other hand for small values of permeability, we obtain formally
$$
\p_x \bq_2 \underset{ k\to 0}{\longrightarrow} 0, \quad \quad \p_x \bu = \frac{G(\bq_2)}{\alpha}. $$
Therefore, in a neighbourhood of the value $P \sim 10^{-8}$, the concentration gradient tends to be constant and for this reason we notice a plateau.
\section{Long time behaviour}
\label{sec:longtime}

This section is devoted to the main mathematical result of this paper concerning the long time asymptotics of solutions to \eqref{five} towards solutions to the stationary system \eqref{stat} as time goes to $+\infty$.
We first state the main result and the assumptions needed. Then, we introduce eigenelements of an auxiliary linear system and its dual problem.
Using these auxiliary functions, we are able to show the convergence when the time variable goes to $+\infty$.
A similar approach was considered in \cite{tesp} following ideas from \cite{perth0}.

\subsection{Statement of the main result}

Before stating the main result, we provide assumptions on the initial and boundary data.

\begin{ass}\label{ass.1}
 We assume that the initial solute concentrations are non-negative and uniformly bounded in $L^\infty(0,L)$ and in 
the total variation~:
\begin{equation}\label{bvcond}
0\leq  u_{1}^{0}, u_{2}^{0}, q_{1}^{0}, q_{2}^{0}, u_{0}^{0} \in BV(0,L) \cap L^{\infty}(0,L).
\end{equation}
\end{ass}

\begin{ass}\label{ass.2}
 The boundary condition of system \eqref{five}  is such that 
\begin{equation}\label{boundcond}
  0 \leq u_b  \in  L^{\infty}(\mathbb{R}^+)\cap L^1_{loc}(\mathbb{R}^+), \quad
  \lim_{t\to + \infty} |u_b - \bu_b| = 0,
\end{equation}
for some constant $\bu_b>0$.
\end{ass}

$BV$  is the space of functions with  bounded variation, 
we notice that such functions have a trace on the boundary (see e.g. \cite{evans}); 
hence the boundary condition $u_2(t,L)=u_1(t,L)$ is well-defined.

\begin{ass}\label{ass.3}
 Regularity and boundedness of $G$.
We assume that the non-linear function modelling active transport in the ascending limb (tube 2) is a bounded and Lipschitz-continuous function on $\mathbb{R}^+$~:
\begin{equation}\label{nonlin}
\forall x\in \mathbb{R}^+,\quad  0 \leq G(q_2) \leq \|G\|_\infty, \quad  0 \leq G'(q_2) \leq \|G'\|_\infty.
\end{equation}
\end{ass}
 We notice that $G$ defined by \eqref{defG} satisfies straightforwardly \eqref{nonlin}.

We now state the main result.
\begin{theo}[Long time behaviour]\label{timeconvergence}
  Under Assumptions \ref{ass.1}, \ref{ass.2} and \ref{ass.3},
  the solution to the dynamical problem \eqref{five} denoted by $\mathbf{u}(t,x)=(u_1, u_2, q_1, q_2, u_0)$ converges as time $t$ goes to $+\infty$ towards $\mathbf{\bu}(x)$, the unique solution to the stationary problem \eqref{stat}, in the following sense
  $$
  \lim_{t\to + \infty} \| \mathbf{u}(t)-\mathbf{\bu} \|_{L^{1}(\Phi)} = 0,
  $$
  with the space
  $$
  L^{1}(\Phi) =\Big\{ \mathbf{u} : [0,L] \rightarrow \mathbb{R}^{5} ; \quad
  \|\mathbf{u}\|_{L^1(\Phi)} := \int_{0}^{L} |\mathbf{u}(x)| \cdot \Phi(x)\ dx < \infty \Big\},
  $$
  where $\Phi=(\ph_1, \ph_2, \phi_1, \phi_2, \ph_0)$ is defined in Proposition \ref{prop} below.
  
  Moreover, if we assume that there exist $\mu_0>0$ and $C_0$ such that
  $|u_b(t)-\bu_b|\leq C_0 e^{-\mu_0 t}$ for all $t>0$,
  then there exist $\mu>0$ and $C>0$ such that we have the convergence with an exponential rate
  \begin{equation}\label{convergence}
    \| \mathbf{u}(t)-\mathbf{\bu} \|_{L^{1}(\Phi)} \leq C e^{- \mu t}.
  \end{equation}
\end{theo}

The scalar product used in the latter claim means : 
  \begin{align*}
    & \int_{0}^{L} |\mathbf{u}(x)| \cdot \Phi(x)\ dx =  \\
    & \int_{0}^{L} \bigl(|u_1|\ph_{1}(x)+|u_2|\ph_2(x)
  +|q_1|\phi_1(x)  +|q_2|\phi_2(x)+|u_0|\ph_{0}(x)\bigr)\ dx.
  \end{align*}
  The definition of the left eigenvector $\Phi$ and its role are given hereafter.

\subsection{The eigen-problem}

In order to study the long time asymptotics of the time dependent system \eqref{five}, we consider the eigen-problem associated with a specific linear system \cite{perth,tesp}.
This system is, in some sort, a linearized version of the stationary system \eqref{stat} where the derivative of the non-linearity is replaced by a constant $g$. 
When these eigenelements $(\lambda, \mathcal{U}, \Phi)$ exist, the asymptotic growth rate in time for a solution $\mathbf{u}$ of \eqref{five} is given by the first positive eigenvalue $\la$ and the asymptotic shape is given by the corresponding eigenfunction $\mathcal{U}$.

Let us introduce the eigenelements of an auxiliary stationary linear system
\begin{equation}\label{directc}
\begin{cases}
\p_x U_1=\lambda U_1 +k(Q_1- U_1) \\
-\p_x U_2=\lambda U_2 +k (Q_2-U_2) \\
0=\lambda Q_1 + k (U_1-Q_1) +K_1 (U_0-Q_1) \\
0=\lambda Q_2 + k(U_2-Q_2)-g Q_2 \\
0=\lambda U_0 + K_1(Q_1-U_0)+g Q_2, 
\end{cases}
\end{equation}
where $g$ is a positive constant which will be fixed later.
This system is complemented with boundary and normalization conditions~:
\begin{equation}\label{stabcond}
U_1(0)=0, \quad U_1(L)=U_2(L), \quad \int_{0}^{L} (U_1+U_2+Q_1+Q_2+U_0)\ dx=1 .
\end{equation}
We also consider the related dual system~:
\begin{equation}\label{dualc}
\begin{cases}
-\p_x \ph_1=\lambda  \ph_1 + k(\phi_1-\ph_1) \\
\p_x \ph_2=\lambda \ph_2 + k(\phi_2-\ph_2) \\
0=\lambda \phi_1 + k(\ph_1-\phi_1) +K_1 (\ph_0-\phi_1) \\
0=\lambda \phi_2 + k(\ph_2-\phi_2)+ g (\ph_0-\phi_2) \\
0=\lambda \ph_0 + K_1(\phi_1-\ph_0),
\end{cases}
\end{equation}
with following conditions~:
\begin{equation}\label{dualccondition}
\ph_1(L)=\ph_2(L), \quad \ph_2(0)=0, \quad \int_{0}^{L} (U_1 \ph_1+U_2 \ph_2+Q_1 \phi_1+Q_2 \phi_2+U_0 \ph_0)\,dx=1. 
\end{equation}
For a given $\lambda$, the function $\mathcal{U}:=(U_1, U_2, Q_1, Q_2, U_0)$ is the right eigenvector solving \eqref{directc}, while $\Phi:=(\ph_1, \ph_2, \phi_1, \phi_2, \ph_0)$ is the left one, associated with the adjoint operator.
The following result shows the existence of a positive eigenvalue and some properties of eigenelements.
We underline that in order to make the proof easier, we consider the case $k=k_1=k_2$ but the same result could be extended to the more general case where $k_1 \ne k_2$. 
\begin{prop}\label{prop}
  Let $g>0$ be a constant. There exists a unique $(\la, \mathcal{U}, \Phi)$ with $\la \in (0,\la_{-})$ solution to the eigenproblem \eqref{directc}--\eqref{dualccondition}, where
  $$ \displaystyle{\la_{-}= \frac{(2 K_1 + k) - \sqrt{4K_1^2+k^2}}{2}}.$$
  Moreover, we have
  $\mathcal{U}(x)>0$, $\Phi(x)>0$ on $(0,L)$ and $\phi_2<\ph_0$. 
\end{prop}
In order to prove this result, we will divide the proof in two steps~:
 Lemmas \ref{lemmadirect} and  \ref{lemmadual} respectively. 
Proposition \ref{prop} is a direct consequence of these two Lemmas.
We start with the direct problem~:
\begin{lem}[The direct problem]\label{lemmadirect}
	There exists a unique $\la>0$ such that the direct problem \eqref{directc}-\eqref{stabcond} admits a unique positive solution    $\mathcal{U}=(U_1, U_2, Q_1, Q_2, U_0)$ on $(0,L)$, and $0<\la<\la_{-}.$
\end{lem}
\begin{proof}
  Summing all equations in \eqref{directc} we find that~:
  \begin{equation}
    U_1'-U_2'=\lambda(U_1+U_2+Q_1+Q_2+U_0) .
  \end{equation} 
  Integrating with respect to $x$ and using condition \eqref{stabcond}, we obtain $U_2(0)=\lambda$.
  By the fourth equation in \eqref{directc}, we find directly:
  \begin{equation}\label{q2} Q_2(x)=\frac{k U_2(x)}{k+g-\lambda} = \frac{U_2(x)}{1+\frac{1}{k}(g-\la)}.
  \end{equation}	
  Putting this expression into the second equation  in \eqref{directc}, we find
  $$
  -U_2'= U_2\left(\la + \frac{\la-g}{1+\frac{1}{k} (g-\la)}\right).
  $$
  Solving the latter equation, we deduce that 
\begin{equation}\label{24bis}
\begin{aligned}
  & U_2(x) =U_2(0)e^{-\la x+\int_{0}^{x} \frac{-\la+g}{1+\frac{1}{k}(g-\la)}\ dy }   = \la e^{(-\la +\eta(\la) )x};  \\
    & \text{with} \; \eta(\la) := \frac{-\la+g}{1+\frac{1}{k}(g-\la)} . 
\end{aligned}
\end{equation}
  Using the fifth equation of system \eqref{directc} we recover
  $$
  U_0(x)=\frac{K_1}{K_1-\la}Q_1(x)+\frac{g}{K_1-\la}Q_2(x).
  $$
  We inject this into the third equation to obtain
  $$
  Q_1(x)\left(k -\la -\frac{K_1 \la}{K_1-\la}\right)=\frac{g K_1}{K_1-\la}Q_2(x)+k U_1(x).
  $$
  Thanks to \eqref{q2} we write also: 
  $$
  Q_1(x)\left(k -\la -\frac{K_1 \la}{K_1-\la}  \right) = \frac{K_1 g }{K_1-\la }\frac{1}{(1+\frac{1}{k}(g-\la))} U_2(x)+k U_1(x).
  $$ 
  Taking into account the first equation of system \eqref{directc}, we obtain~:
  \begin{equation}\label{deruuno}
    U_1'(x)=c_{\la} U_1(x) + k_{\la} \frac{g}{1+\frac{1}{k} (g-\la )} U_2(x),
  \end{equation}	
  where we simplify notations by introducing~:
  \begin{equation}\label{constant}
    k_{\la}:=\frac{k\frac{K_1}{K_1-\la}}{k -\la - \frac{K_1 \la}{K_1-\la}}, \quad c_{\la}:=\la+ \frac{k(\la+\frac{K_1 \la}{K_1-\la})}{k-\la - \frac{K_1 \la}{K_1-\la}  }.
  \end{equation}
  The denominator $k-\la - \frac{K_1 \la}{K_1-\la}$ vanishes for 
  $$
  \la_{\pm}= \frac{(2 K_1 + k) \pm \sqrt{4K_1^2+k^2}}{2}.
  $$
  Obviously $\lim_{\la \rightarrow \la_{-}} k_{\la}=+\infty$ and we also have that  $0<\la_{-}< \min(K_1, k)$.
	
  Now we solve directly the ODE \eqref{deruuno} with its initial condition, we get
  $$
  U_1(x)=\frac{\la g k_{\la} }{1+\frac{1}{k}(g-\la)} \frac{e^{c_{\la}x}- e^{(\eta(\la)-\la) x} }{c_\lambda + \lambda - \eta(\lambda)}.
  $$
  We are looking for a $\la>0$ such that boundary condition $U_1(L)=U_2(L)$ is satisfied, in other words $\frac{U_1(L)}{U_2(L)}=1$, namely
  \begin{equation}
    \label{flambda}
  F(\la):=\frac{g k_{\la}}{ 1+\frac{1}{k}(g-\la) } \left( \frac{e^{(c_{\la}+\la-\eta(\la))L} - 1}{c_{\la } +\la -\eta(\la)}  \right)=1,
  \end{equation}
  where we recall that $k_{\la}$, $c_{\la}$ are defined in \eqref{constant} and $\eta(\la)$ in \eqref{24bis}.
  We remark immediately that for $\la=0$ in \eqref{constant}, we have   $k_{0}=1$, $c_{0}=0$. Then,
  $$
  F(0)= 1-\exp\Bigl(-\frac{gL}{1+\frac{g}{k} }  \Bigr) <1.
  $$
  We notice that for $k_{\la}, c_{\la}>0$, $F(\la)$ is a continuous increasing function with respect to $\la$ since the product of increasing and positive functions is still increasing (see Appendix \eqref{appendix:function} for more details).
  Moreover $\lim_{\la \rightarrow \la_{-}} F(\la) = +\infty$.
  Then it exists a unique $\la\in (0,\la_-)$ such that $F(\la)=1$. Moreover, for $0<\la<\la_{-}<\min(k,K_1) $, the functions $U_1,U_2,Q_1,Q_2,U_0$ are positive on $[0,L]$.
	
\end{proof}

\begin{lem}[The dual problem]\label{lemmadual}
  Let $\la$ and $\mathcal{U}$ be as in Lemma \eqref{lemmadirect}.
  Then, there exists $\Phi:=(\ph_1, \ph_2, \phi_1, \phi_2, \ph_0)$, the unique solution of dual problem \eqref{dualc}--\eqref{dualccondition} with $\ph_1, \ph_2, \phi_1, \phi_2, \ph_0 > 0 $.  Moreover, we have $\phi_2 < \ph_0$.
\end{lem}

\begin{proof}
  By the fifth equation of system \eqref{dualc} we have directly~:
  $$
  \ph_0=\frac{K_1}{K_1-\la}  \phi_1.
  $$
  Replacing this expression in the third equation we obtain
  $$
  (k- \la - \frac{K_1 \la}{K_1-\la}) \phi_1=k \ph_1.
  $$
  Then,
  $$
  \ph_0(x)=\frac{k \frac{K_1}{K_1-\la}}{k -\la - \frac{K_1 \la}{K_1-\la}} \ph_1(x)= k_{\la} \ph_1(x),
  $$
  where   $k_\la$ is defined in \eqref{constant}.
  Using the first equation of \eqref{dualc}, we have
  $$ 
  -\ph_1'=\ph_1\left( \la+k \left(\frac{\la +\frac{K_1\la}{K_1-\la}}{k -\la -\frac{K_1 \la}{K_1-\la}} \right) \right).
  $$
  Integrating, we obtain
  \begin{equation}\label{phi1}
    \ph_1(x)=\ph_1(0)e^{-\la x}e^{-\beta x}, \quad \beta=\beta_{\la}=\frac{ \la k(2K_1-\la)}{\la^2 -2 K_1 \la   -\la k+K_1 k}.
  \end{equation}
  We  easily check that $\beta>0$ if $0<\la<\la_{-}<K_1$.

 As shown in details in the Appendix \ref{app.B2}, 
  for all $x\in (0,L)$, $(U_1 \ph_1)'-(U_2\ph_2)'=0$.
  Integrating, we get $U_1(x) \ph_1(x) - U_2(x) \ph_2(x) = U_1(0) \ph_1(0) - U_2(0) \ph_2(0) = 0$, thanks to boundary conditions $U_1(0)=0$ and $\ph_2(0)=0$. (Notice also that taking $x=L$ in this latter relation, and using the boundary condition $U_1(L)=U_2(L)\ne 0 $, we recover $\ph_1(L)=\ph_2(L)$.)
  Therefore, we get
  \begin{equation}\label{relationc}
    \ph_2(x)=\frac{U_1(x)}{U_2(x)} \ph_1(x), \quad \forall x \in [0,L].
  \end{equation}	
  Using the fourth equation in \eqref{dualc} and thanks to \eqref{relationc}, we obtain
  $$
  0=\la \phi_2 +k \frac{U_1}{U_2} \ph_1 -k \phi_2 +
  g\ph_1 \left( \frac{K_1}{K_1-\la}\cdot\frac{k}{k -\la -\frac{K_1 \la}{K_1-\la}}  \right),
  $$
  which allows  to compute $\phi_2$~:
  $$
  \phi_2(x)=\frac{k \frac{U_1(x)}{U_2(x)} \ph_1(x)}{k -\la + g} + \frac{g}{k -\la +g}k_{\la} \ph_1(x).
  $$
  Each function depends on the first component of $\Phi$, i.e. $\ph_1(x)$, and to sum up, the following relation has been obtained:
  \begin{equation}\label{dualcsolution}
    \begin{cases}
      \ph_1(x)=\ph_1(0)e^{-\la x} e^{-\beta x} \\
      \ph_2(x)=\frac{U_1(x)}{U_2(x)} \ph_1(x) \\
      \phi_1(x)=\ph_1(x)\left(\frac{k}{k-\la -\frac{K_1 \la}{K_1-\la}}\right) \\
      \phi_2(x)=\frac{1}{k -\la + g}\left(k \frac{U_1(x)}{U_2(x)} +g k_{\la}  \right)\ph_1(x) \\
      \ph_0(x)=k_{\la} \ph_1(x),
    \end{cases}
  \end{equation}
  where  $k_{\la}, \beta$ are defined  in \eqref{constant} and \eqref{phi1}.
  Hence the sign of $\Phi$ depends on the sign of $\ph_1(0)$, 
  the other quantities and constants  being positive for $\la \in (0,\la_{-})$ and $g>0$ by assumption.
  Then, we use the normalization condition \eqref{dualccondition} and \eqref{dualcsolution} in order to show the positivity of $\ph_1(0)$.
It implies that
  \begin{equation*}
    \begin{split}
      \ph_1(0)  \int_{0}^{L}  e^{-\la x} e^{-\beta x} \left[ 2U_1(x) +Q_1(x)\left(\frac{k}{k-\la -\frac{K_1 \la}{K_1-\la}}\right) + \right. & \\
      \left. + \frac{Q_2(x)}{k +g+\la}\left(k \frac{U_1(x)}{U_2(x)} +g  k_{\la} \right)  +k_{\la} U_0(x) \right]\ dx=1. & \\
    \end{split}
  \end{equation*}
   The integral on the left hand side is positive,  thanks to properties of functions previously defined.
   Given that $\ph_1(0)$ is constant and all other quantities positive, we can conclude that $\ph_1(0)>0$. 
 

  We are left to prove that the quantity $\phi_2-\ph_0$ is negative.
  Using  \eqref{dualcsolution}, we rewrite~:
  $$
  \phi_2-\ph_0=  \frac{k k_{\la} \ph_1(x)}{k-\la+g}\left(\frac{1}{k_{\la}}\frac{U_1}{U_2 }+\frac{\la}{k} -1\right).
  $$
  From the explicit expression of $U_1$ and $U_2$, we have
  \begin{align*}
    \phi_2-\ph_0 & = \frac{ k_{\la} \ph_1(x)}{1+\frac{1}{k}(g-\la)} \left[ \frac{\int_{0}^{x} \frac{g}{1+\frac{1}{k}(g-\la)} e^{-c_{\la}(y-x)} e^{(-\la+\eta(\lambda)) y}\,dy}{e^{(-\la+\eta(\la)) x}} +\frac{\la}{k} -1 \right]  \\
                 & = \frac{ k_{\la} \ph_1(x)}{1+\frac{1}{k}(g-\la)}\left[\frac{g}{1+\frac{1}{k}(g-\la)} \Big[\frac{1-e^{-c_\la x-\la x+\eta(\lambda) x}}{c_{\la}+\la - \eta}\Big] +\frac{\la}{k} -1 \right],
  \end{align*}
  where we recall the notation $\eta(\la)=\frac{-\la +g}{1+\frac{1}{k} (g-\la)}$.
We set
\begin{equation}
H(x):= \frac{g}{1+\frac{1}{k}(g-\la)} \Big[\frac{1-e^{-c_\la x-\la x+\eta(\lambda) x}}{c_{\la}+\la - \eta(\la)}\Big].
\end{equation}
We have
$$
\phi_2-\ph_0 < 0 \iff H(x)+\frac{\la}{k} -1<0.
$$
We observe that $H(0)=0$ and $H(L)=\frac{1}{k_{\la}}$ thanks to  \eqref{flambda}. Moreover, $H(L)<1-\frac{\la}{k} $  for $\la \in (0,\la_{-})$. Indeed, we have
$$
  \la < k-\frac{k}{k_{\la}},
$$ 
which holds if and only if $ \frac{\la^2 -K_1\la-k\la}{K_1}<0 $ which is true on $(0,\la_{-})$,  since $\la_{-}<k$ by definition.
Moreover, it is clear that $H$ is an increasing function on $[0,L]$ for $\la \in (0,\la_{-})$. Then $H(x) \leq H(L) < 1 - \frac{\lambda}{k}$. This concludes the proof.
\end{proof}

\subsection{Proof of Theorem \ref{timeconvergence}}

Now we are ready to prove Theorem \ref{timeconvergence}.
We set $d_i(t,x):=|u_i(t,x)-\bar{u}_{i}(x)|$ $i=0,1,2$ and $\de_{j}:=|q_j(t,x)-\bar{q}_{j}(x)|$, $j=1,2$ with $\bar{u}_i, \bar{q}_i$ satisfying \eqref{stat} and $u_i, q_i$ solving \eqref{five}.

We subtract  component-wise  \eqref{five} to \eqref{stat}. Then  we multiply each of the entries by $sign(u_i-\bu_i)$ or $sign(q_j-\bq_j)$ respectively.
We obtain the following inequalities~:
\begin{equation}\label{inequalities}
\begin{cases}
a_1\p_t d_1+\alpha\p_x d_1\leq k(\de_1-d_1) \\
a_2\p_t d_2 - \alpha\p_x d_2\leq k(\de_2-d_2) \\
a_3\p_t \de_1\leq k (d_1-\de_1) +K_1 (d_0-\de_1) \\
a_4\p_t \de_2\leq k(\de_2-d_2)-\hat{G} \\
a_0\p_t d_0\leq K_1(\de_1-d_0) + \hat{G},
\end{cases}
\end{equation}
with $\hat{G}:=|G(q_2)-G(\bq_2)|$.
We have used also the monotonicity of $G$ (see \eqref{nonlin}).
We set
$$
M(t):=\int_{0}^{L} (a_1 d_1 \ph_1 +a_2 d_2 \ph_2 +a_3 \de_1 \phi_1 +a_4 \de_2 \phi_2 + a_0 d_0 \ph_0 )\,dx.
$$

Multiplying each equation of \eqref{inequalities} by the corresponding dual function $\ph_i, \phi_i$, adding all equations and integrating with respect to $x$, we obtain~:
\begin{multline*}
\frac{d}{dt} M(t) \leq \int_{0}^{L} \Bigl(k(\de_1-d_1)\ph_1+k (\de_2-d_2)\ph_2 + k (d_1-\de_1)\phi_1 +K_1 (d_0-\de_1) \phi_1\\ 
+k(\de_2-d_2)\phi_2 -\hat{G} \phi_2  
+ K_1(\de_1-d_0)\ph_0 + \hat{G} \ph_0 \Bigr)\ dx + \alpha\int_{0}^{L} (\p_{x} d_2 \ph_2-\p_x d_1 \ph_1)\,dx .
\end{multline*}
Integrating by parts the last integral and using  the dual system \eqref{dualc}, 
we can simplify the latter inequality into
\begin{multline*}
\frac{d}{dt} M(t) \leq -\lambda \int_{0}^{L} (d_1 \ph_1 +d_2 \ph_2 + \de_1 \phi_1 + \de_2 \phi_2 + d_0 \ph_0 )\ dx   \\
+d_2(L)\ph_2(L)-d_1(L)\ph_1(L)-d_2(0)\ph_2(0)+d_1(0)\ph_1(0)+ \int_{0}^{L} (g\de_2 -\hat{G})(\phi_2-\ph_0)\ dx. 
\end{multline*}
Using the normalization conditions in \eqref{stabcond}  and in  \eqref{dualccondition}, we obtain
$$
\frac{d}{dt} M(t)\leq \frac{-\lambda}{\max\{ a_1, a_2, a_3, a_4, a_0 \}} M(t) +d_1(t,0)\ph_1(0)+ \int_{0}^{L} (g \de_2-\hat{G})(\phi_{2}-\ph_0)\ dx.
$$ 
To simplify the notation we set $\bar{\lambda}=\frac{-\lambda}{\max\{ a_1, a_2, a_3, a_4, a_0 \}}$.
Since $G$ is  Lipschitz-continuous  and by assumption \eqref{nonlin},  $\hat{G} \leq g \de_2$ with $g=\|G'\|_{\infty}$.
With this choice of $g$, we apply Proposition \ref{prop} and deduce that the quantity $(\phi_2-\ph_0)$ is negative.
Then,
$$
\frac{d}{dt} M(t)+\bar{\lambda} M(t) \leq d_{1}(t,0)\ph_1(0).
$$
Thanks to \eqref{boundcond} and applying Gronwall's lemma, we conclude that
%
\begin{equation}
  \label{eq:M1}
  M(t)\leq M(0)e^{-\bar{\la} t} + \ph_1(0)\int_{0}^{t} d_1(s,0)e^{\bar{\la} (s-t)}\,ds.  
\end{equation}
Moreover, from \eqref{boundcond}, we have
$d_1(s,0)=|u_b(t)-\bu_b|\rightarrow 0 $ as $t \rightarrow +\infty$.
Then, for every $\eps > 0$, it exists $\bar{t}>0$ such that $d_1(s,0)<\eps$ for each $s>\bar{t}$.
Then for every $t\geq \bar{t}$, we have
\begin{align*}
  \int_{0}^{t} d_1(s,0)e^{\bar{\la} (s-t)}\ ds & \leq \int_{0}^{\bar{t}} d_1(s,0)e^{\bar{\la}(s-t)}\ ds + \eps \int_{\bar{t}}^{t} e^{\bar{\la}(s-t)}\,ds \\
  & \leq e^{\bar{\la} (\bar{t} - t)}\int_{0}^{\bar{t}} d_{1}(s,0)\ ds + \frac{\eps}{\bar{\la}}.
\end{align*}
The first term of the right hand side is arbitrarily small at $t$ goes to $+\infty$.
Hence, we have proved that for any $\eps>0$ there exists $\tau$ large enough such that for every $t\geq \tau$,
$$
M(t)\leq M(0)e^{-\bar{\la} t} + C\eps.
$$
Since $M(t) = \|\mathbf{u}(t)-\bar{\mathbf{u}}\|_{L^1(\Phi)}$, it proves the convergence as stated in Theorem \ref{timeconvergence}.

Finally, if we assume that there exist positive constants $\mu_0$ and $C_0$ such that $|u_b(t)-\bu_b|\leq C_0 e^{-\mu_0 t}$, then from \eqref{eq:M1} we deduce
$$
M(t) \leq M(0) e^{-\bar{\la} t} + C_0 \ph_1(0) \frac{e^{-\mu_0 t} - e^{-\bar{\la} t}}{\bar{\la} - \mu_0} \leq C e^{-\min\{\bar{\la}\label{key},\mu_0\} t}.
$$
\qed

\section{Conclusion and outlook}

In this study we present a model describing the transport of sodium in a simplified version of the loop of Henle in a kidney nephron. 
From a modelling point of view, it seems important to take into account the epithelium in the counter-current tubular architecture since we observe that it may affect strongly the solute concentration profiles for a particular range of permeabilities.

The main limitation of the model is to not consider the re-absorption of water in descending limb. 
Indeed, in Section \ref{sec:stat}, we study the steady state solution and the assumption of a constant rate  $\alpha$ 
and the boundary conditions lead to $\bu_1(x)=\bu_2(x)$, 
i.e. the luminal concentrations of sodium are the same in both tubules for every $x\in(0,L)$. 
Conversely, in vivo, the concentrations in lumen 1 and 2 are different 
due to the constitutive differences between the segments and presence of membrane channel proteins, for example the aquaporins.
The thin descending limb of Henle's loop has low permeability to ions and urea, 
while being highly permeable to water. The thick ascending limb is impermeable to water, but it is permeable to ions.
For this reason, a possible extension of the model shall assume that $\alpha$ is not  constant  
but space-dependent. A first step could be, for instance,  to take two different values of $\alpha$ for the first and second equation of the model \eqref{five}, $\alpha_1$ and $\alpha_2$. From the mathematical viewpoint, this choice slightly changes the structure of the hyperbolic system~: 
for example,    conservation of certain quantities should not be that easy to prove. 

Furthermore, this assumption about $\alpha$ has a relevant influence on other factors. As already pointed out, the relation between $\bq_1$ and $\bu_0$ is biologically correct and consistent, this means that in vivo the concentration of Na$^{+}$ in the epithelial cell (intracellular) is lower than in interstitium.  
 The intracellular concentrations (epithelium,
$\bq_1$ and $\bq_2$) are usually of the order of $10$mM whereas the extracellular ones (therefore in the lumen and in the interstitium) are of the order of $140$mM, (see \cite{feher}, page 692).\\
There are also other types of source terms  in the interstitium that could be  added, accounting for blood vessels and/or collecting ducts. 
In this case, the last equation \eqref{fivee} of the dynamic system  should include a term that accounts for interstitium concentration storage or accumulation and for secretion-reabsorption  of water and solutes, 
but  the impact of adding such  complex mechanisms in the model remains to be assessed.

In this study, we focused our attention on the axial concentration gradient and the FIC, previously defined in Section \eqref{sec:stat}, which are 
significant factors in the urinary concentration mechanism, \cite{layton, laytonedwards}.
The axial gradient is an important determinant of urinary concentration capacity.
When water intake is limited, mammals can conserve water in body fluids by excreting solutes in a reduced volume of water, that is, by producing a concentrated urine. The thick ascending limb plays an essential role in urine concentration and dilution, \cite{sands}: the active reabsorption of sodium without parallel reabsorption of water generates an interstitial concentration gradient in the outer medulla that in turn drives water reabsorption by the collecting ducts, thereby regulating the concentration of final urine. \\
In summary, our model confirms that the active trans-epithelial transport of Na from the ascending limbs into the surrounding environment is able to generate an osmolality gradient.
Our model indicates that explicitly accounting for the 2-step transport across the epithelium significantly impacts the axial concentration gradient within the physiological range of parameters values considered here. Thus, representing the epithelial layer as two membrane in series, as opposed to a single-barrier representation, may provide a more accurate understanding of the forces that contribute to the urinary concentrating mechanism.



\appendix

\section{Large permeability asymptotic}\label{formalcomputation}
In this section we consider the case where the permeability between the lumen and the epithelium is large, 
i.e. when $P_i \rightarrow \infty $, with $i=1,2$ in the definition of constants $k_1$ and $k_2$.
For this purpose,  we set $k=k_1=k_2= \frac 1\eps$ and we let $\eps$ go to $0$.
Physically, this means fusing the epithelial layer with the lumen.

Rewriting   \eqref{five}  in this perspective gives
\begin{subequations}\label{fivea}
  \begin{equation}\label{fiveaa}
    \partial_{t}u_1^{\eps} + \alpha\partial_{x} u_1^{\eps} = \frac 1\eps (q_1^{\eps}-u_1^{\eps}) 
	\end{equation} 
	\begin{equation}\label{fiveba}
	\partial_{t}u_2^{\eps} - \alpha\partial_{x} u_2^{\eps} =  \frac 1\eps (q_2^{\eps}-u_2^{\eps}) 
	\end{equation} 
	\begin{equation}\label{fiveca}
	\partial_{t}q_1^{\eps} = \frac 1\eps (u_{1}^{\eps}-q_1^{\eps}) + K_1(u_{0}^{\eps}-q_1^{\eps}) 
	\end{equation} 
	\begin{equation}\label{fiveda}
	\partial_{t}q_2^{\eps} = \frac 1\eps (u_{2}^{\eps}-q_2^{\eps}) - G(q_2^{\eps})
	\end{equation} 
	\begin{equation}\label{fiveea}
	\partial_{t}u_0^{\eps} = K_1(q_1^{\eps}-u_0^{\eps})  + G(q_2^{\eps}) .
	\end{equation}
\end{subequations}
We expect the concentrations $u_1^{\eps}$ and $q_1^{\eps}$ to converge to the same quantity. 
The same happens for $u_2^{\eps} \rightarrow u_2$ and $q_2^{\eps} \rightarrow u_2$.
We denote $u_1$, respectively $u_2$, the limit of $u_1^\eps$ and $q_1^\eps$, 
respectively $u_2^\eps$ and $q_2^\eps$. 
Adding  \eqref{fiveaa} to \eqref{fiveca} 
and  \eqref{fiveba} to \eqref{fiveda}, we obtain 
\begin{align*}
\p_{t} u_1^{\eps}+ \p_{t}q_1^{\eps}+ \alpha\p_{x} u_1^{\eps}= &\ K_1(u_{0}^{\eps}-q_1^{\eps}) \\
\p_{t} u_2^{\eps}+ \p_{t}q_2^{\eps}- \alpha\p_{x} u_2^{\eps}= &\  - G(q_2^{\eps}).
\end{align*}
Passing formally to the limit $\eps \rightarrow 0$, we arrive at
\begin{align}
2\p_{t} u_1+ \alpha\p_{x} u_1= &\ K_1(u_{0}-u_1)   \label{systlim1} \\
2\p_{t} u_2-\alpha\p_{x} u_2 = &\ - G(u_2), \label{systlim2}
\end{align}
coupled to the equation for the concentration in the interstitium obtained by passing into the limit in equation \eqref{fiveea}
\begin{equation}\label{systlim3}
\p_{t} u_0=K_1(u_1-u_0)  + G(u_2).
\end{equation}
The equations \eqref{systlim1}, \eqref{systlim2}, \eqref{systlim3} describe the same concentration dynamics in a system without epithelium, previously studied in \cite{magali} and \cite{tesp}. 
The formal computation above shows that this $3\times 3$ system may be considered as a good approximation of the larger system \eqref{five} for large permeabilities.

Such a convergence result may be proved rigorously and it is investigated in \cite{MMV}. 
It relies on specific {\em a priori} estimates and the introduction of an initial layer.


%
%

%
%
%
%

\section{Technical results}

\subsection{Function $F(\la)$}\label{appendix:function}

In this subsection we prove the monotonicity of the function $F(\la)$ which appears in the proof of Lemma \ref{lemmadirect}.
First let's recall it
\begin{equation}\label{eqflambda}
  F(\la):=\frac{g k_{\la}}{ 1+\frac{1}{k}(g-\la) } \left( \frac{e^{(c_{\la}+\la-\eta(\la))L} - 1}{c_{\la } +\la -\eta(\la)}  \right).
\end{equation}
\begin{lem}
The function $F$ defined by \eqref{eqflambda} is monotonically increasing on $(0,\la_-)$.
\end{lem}
\begin{proof}
 The product of positive increasing functions is increasing.
  \begin{itemize}
  \item $\la\mapsto k_{\la} = \frac{K_1 k}{\la^2-2K_1\la - k \la  +k K_1}$ is a positive and increasing function if $\la \in (0,\la_{-})$.
    Indeed $\frac{\p k_\la}{\p \la}=\frac{-2\la k K_1+2k K_1^2  +k^2K_1}{(\la^2-2K_1\la - k \la  +k K_1)^2}$
    is positive for $0<\la<K_1+\frac{k}{2}$ and $\la_{-}<K_1 $ by definition.
  \item We set $f_1(\la):=\frac{g}{1+\frac{1}{k}(g-\la)}$; if $\la< g+k $ the function $f_1$ is positive since $g>0$ by hypothesis and it is also increasing since $\frac{\p}{\p_{\la}} f_1(\la)= \frac{\frac{g}{k}}{(1+\frac{1}{k}(g(y)-\la))^2}> 0$, and $\lambda_- \leq \frac{k}{2}$.
  \item The function $x\mapsto \frac{e^x-1}{x}$ is increasing on $\mathbb{R^+}$ and the function $\lambda\mapsto c_\la +\la - \eta(\la)$ is increasing on $(0,\lambda_1)$.
    Indeed, we have straightforwardly
    $$
    c_\la +\la - \eta(\la) = 2 \la + 2 k + \frac{k^2}{k-\la- \frac{K_1\la}{K_1-\la}} + \frac{k^2}{k+g-\la}.
    $$
\end{itemize}
\end{proof}

\subsection{Relation between direct and dual system}\label{app.B2}
We recall the eigenelements problem written as below:
\begin{equation}\label{relation1}
 \begin{bmatrix} 
 \p_x U_1(x) \\ -\p_x U_2(x) \\ 0 \\ 0 \\ 0 \\
 \end{bmatrix} = \la \mathcal{U}(x)+A\mathcal{U}(x);
  \quad\quad \mathcal{U}(x)=\begin{bmatrix} 
U_1 \\ U_2 \\ Q_1 \\ Q_2 \\ U_0
\end{bmatrix} 
\end{equation}

\begin{equation}\label{eqeigen}
\begin{bmatrix} 
-\p_x \ph_1(x) \\ \p_x \ph_2(x) \\ 0 \\ 0 \\ 0
\end{bmatrix} = \la \Phi(x)+ {}^tA\Phi(x);
\quad\quad \Phi(x)=\begin{bmatrix} 
\ph_1 \\ \ph_2 \\ \phi_1  \\ \phi_2 \\ \ph_0
\end{bmatrix} 
\end{equation}
with related matrix defined by
$$
A=\begin{bmatrix}
-k & 0 & k & 0 & 0 \\
0  & -k & 0 & k & 0 \\
k & 0 & -k-K_1 & 0 & K_1 \\
0 & k & 0 & -k-g & 0 \\
0 & 0 & K_1 & g & -K_1 
\end{bmatrix}.
$$
Multiplying \eqref{relation1} on the left by ${}^t\Phi$, we deduce
$$
\ph_1 \p_x U_1 - \ph_2 \p_x U_2 = \lambda {}^t\Phi  \ \mathcal{U} + {}^t\Phi A \ \mathcal{U}.
$$
Taking the transpose of \eqref{eqeigen} and multiplying on the right by $\mathcal{U}$, we also have
$$
- \p_x \ph_1  U_1 + \p_x \ph_2  U_2 = \lambda {}^t\Phi \ \mathcal{U} + {}^t\Phi A \ \mathcal{U}.
$$
As a consequence, we deduce the relation
\begin{align}\label{relation}
 (U_1 \ph_1)'-(U_2 \ph_2)'=0,  \quad  \forall x \in [0,L]. 
\end{align}

 Since  $U_1(L)=U_2(L)$ in \eqref{stabcond} and by initial conditions $U_1(0)=0, \ \ph_2(0)=0$, then also $\ph_1(L)=\ph_2(L)$, as set in \eqref{dualccondition}. It means that $(U_1\ph_1)=(U_2\ph_2)  \ \forall x \in [0,L]$. Thanks to this relation, we can consider in our previous computation:
 $$\ph_2(x)=\frac{U_1(x)}{U_2(x)} \ph_1(x), \quad \forall x \in [0,L].  $$

\bigskip

\footnotesize{
 
}

\end{document}